\documentclass[12pt,reqno]{amsart}
\usepackage{amsmath,amsthm,amsfonts,amssymb,amscd,amstext}
\usepackage[ansinew]{inputenc}
\usepackage[dvips]{graphicx}
\usepackage{psfrag}
\usepackage{euscript}
\usepackage{a4wide}

\usepackage{pslatex}
\numberwithin{equation}{section}

\usepackage[active]{srcltx}
\usepackage[dvips]{color}
\usepackage[colorlinks=true,linkcolor=red,urlcolor=black]{hyperref}

\theoremstyle{plain}

\newtheorem{theorem}{Theorem}[section]
\newtheorem{proposition}[theorem]{Proposition}
\newtheorem{corollary}[theorem]{Corollary}
\newtheorem{lemma}[theorem]{Lemma}

\theoremstyle{definition}
\newtheorem{remark}[theorem]{Remark}
\newtheorem{definition}{Definition}

\newtheorem{example}{Example}

\newcommand{\RR}{{\mathbb R}}

\newcommand{\ov}{\overline}

\newcommand{\la}{\lambda}

\renewcommand{\epsilon}{\varepsilon}

\newcommand{\diag}{\operatorname{diag}}

\newcommand{\indi}{\operatorname{ind}}

\newcommand{\cri}{\operatorname{Crit}}

\newcommand{\sing}{\mathrm{Sing}}

\newcommand{\cC}{\EuScript{C}}

\newcommand{\J}{\EuScript{J}}

\newcommand{\Mundo}{\mathfrak{X}^{1}(M)}

\begin{document}

\title[Singular hyperbolicity and Sectional Lyapunov exponents of various orders] {Singular hyperbolicity and Sectional Lyapunov exponents of various orders}


\thanks{
  L.S. is partially supported by a Fapesb-JCB0053/2013, PRODOC-UFBA2014, CNPq postdoctoral schoolarship at Universidade Federal do Rio de Janeiro. She also thanks to A. Hammerlindl for fruitful conversations during International Conference Dynamics Beyond Uniform Hyperbolicity - Provo-UT 2017.
}

\subjclass{Primary: 37D30; Secondary: 37D25.}
\renewcommand{\subjclassname}{\textup{2000} Mathematics
  Subject Classification}
\keywords{Dominated splitting,
  partial hyperbolicity, sectional hyperbolicity,
  Lyapunov function.}


\author{Luciana Salgado}

\address{Universidade Federal da Bahia\\ Instituto de Matem\'atica\\
Avenida Adhemar de Barros, s/n, Ondina, 40170-110, Salvador, Bahia, Brazil\\
Email: lsalgado@ufba.br, lsalgado@im.ufrj.br}

\begin{abstract}
It is given notions of singular hyperbolicity and sectional Lyapunov exponents of orders beyond the classical ones, namely, other dimensions besides the dimension 2 and the full dimension of the central subbundle
  of the singular hyperbolic set.
   It is obtained a characterization of dominated splittings, partial and singular hyperbolicity in this broad sense, by using Lyapunov exponents and the notion of infinitesimal Lyapunov functions.
  Furthermore, it is given alternative requirements to obtain singular hyperbolicity.
  As an application we obtain some results related to singular hyperbolic sets for flows.
\end{abstract}

\maketitle

\section{Introduction and statement of results}
\label{sec:int-stat}
\hfill

Let $M$ be a compact $C^{\infty}$ riemannian $n$-dimensional manifold, $n \geq 3$.
Let $\Mundo$ the set of $C^1$ vector fields on $M$, endowed with the $C^1$ topology.
And denote by $X_t: M \to M$ the $C^1$ flow generated by $X$.

In a remarkable work Morales, Pac\'ifico and Pujals \cite{MPP99} defined the so called \emph{singular hyperbolic systems}, in order to describe the behaviour of Lorenz attractor.
It is an extension of the hyperbolic theory for invariant sets for flows which are not (uniformly) hyperbolic,
but which have some robust properties, certain kind of weaker hyperbolicity and also admit singularities.
In \cite{MPP04}, the same authors proved that every robustly transitive singular set for a three dimensional flow
is a partially hyperbolic attractor or repeller and the singularities in this set must be Lorenz-like.

In this paper, we prove a relation between the $\J$-algebra of Potapov \cite{Pota60,Pota79,Wojtk01}
a new definition of singular hyperbolicity, envolving intermediate dimensions of the central subbundle.

The $\J$-algebra here means a pseudo-euclidean structure given by $C^1$ non-degenerate quadratic form $\J$,
defined on $\Lambda$, which generates positive and negative cones with maximal dimension $p$ and $q$, respectively.

The maximal dimension of a cone in $T_xM$ is the maximal dimension of the subspaces contained in there.

   We are going to prove sufficient and necessary conditions for a flow to be singular hyperbolic of some order,
   in a sense to be clarified below.

  It is also given a characterization of singular and sectional hyperbolicity for a flow over a compact invariant set, improving a result in \cite{ArSal2012}.

The text is organized as follow. In first section, it is given the main definitions and stated the results.
In second section, it is presented the main tools by using the notion of $\J$-algebra of Potapov.
In third section, it is proved the main theorems.

\subsection{Preliminary definitions and Main results}
\label{sec:prelim-definit}
\hfill

Before presenting the main statements, we give some definitions.

Let $M$ be a connected compact finite $n$-dimensional manifold, $n \geq 3$,
with or without boundary. We consider a vector
field $X$, such that $X$ is inwardly transverse to the
boundary $\partial M$, if $\partial M\neq\emptyset$. The flow generated by $X$ is denoted by $X_t$.

An \emph{invariant set} $\Lambda$ for the flow of $X$ is a
subset of $M$ which satisfies $X_t(\Lambda)=\Lambda$ for all
$t\in\RR$.  The \emph{maximal invariant set} of the flow is
$M(X):= \cap_{t \geq 0} X_t(M)$, which is clearly a compact
invariant set.

A \emph{singularity} for the vector field $X$
is a point $\sigma\in M$ such that $X(\sigma)=0$ or,
equivalently, $X_t(\sigma)=\sigma$ for all $t \in \RR$. The
set formed by singularities is the \emph{singular set of
  $X$} denoted $\sing(X)$ and $Per(X)$ is the set of periodic points of $X$.
  We say that a singularity \emph{is hyperbolic} if the eigenvalues of the derivative
$DX(\sigma)$ of the vector field at the singularity $\sigma$
have nonzero real part. The set of critical elements of $X$ is
the union of the singularities and the periodic orbits of $X$, and will be denoted by $\cri(X)$.

We recall that a \emph{hyperbolic} set for a flow $X_t$ is an
invariant subset of $\Lambda \subset M$ with a decomposition $T_\Lambda M= E^s\oplus E^X \oplus E^u$
of the tangent bundle which is a continuous splitting,
where $E^X$ is the direction of the vector field, the
subbundles are invariant under the derivative $DX_t$ of the flow
\begin{align*}
  DX_t\cdot E^*_x=E^*_{X_t(x)},\quad
  x\in\Lambda, \quad t\in\RR,\quad *=s,X,u;
\end{align*}
$E^s$ is uniformly contracted by $DX_t$ and $E^u$ is
uniformly expanded: there are $K,\lambda>0$ so that
\begin{align}\label{eq:def-hyperbolic}
\|DX_t\mid_{E^s_x}\|\le K e^{-\lambda t},
  \quad
  \|DX_{-t} \mid_{E^u_x}\|\le K e^{-\lambda t},
  \quad x\in\Lambda, \quad t\in\RR.
\end{align}

Recall that the index of a hyperbolic periodic orbit of a flow is the dimension of the contracting subbundle
of its hyperbolic splitting.

Our main results is the following.

Let $\Lambda \subset M$ be a compact invariant subset for $X$.

\begin{definition}\label{def1}

  A \emph{dominated splitting} over a compact invariant set $\Lambda$ of $X$
  is a continuous $DX_t$-invariant splitting $T_{\Lambda}M =
  E \oplus F$ with $E_x \neq \{0\}$, $F_x \neq \{0\}$ for
  every $x \in \Lambda$ and such that there are positive
  constants $K, \lambda$ satisfying
  \begin{align}\label{eq:def-dom-split}
    \|DX_t|_{E_x}\|\cdot\|DX_{-t}|_{F_{X_t(x)}}\|<Ke^{-\la
      t}, \ \textrm{for all} \ x \in \Lambda, \ \textrm{and
      all} \,\,t> 0.
  \end{align}
\end{definition}

A compact invariant set $\Lambda$ is said to be
\emph{partially hyperbolic} if it exhibits a dominated
splitting $T_{\Lambda}M = E \oplus F$ such that subbundle
$E$ is uniformly contracted. In this case $F$ is the
\emph{central subbundle} of $\Lambda$.

A compact invariant set $\Lambda$ is said to be
\emph{singular-hyperbolic} if it is partially hyperbolic and
the action of the tangent cocycle expands volume along the
central subbundle, i.e.,
\begin{align}\label{eq:def-vol-exp}
      \vert \det (DX_t\vert_{F_x}) \vert > C e^{\la t},
      \forall t>0, \ \forall \ x \in \Lambda.
    \end{align}

The following definition was given as a particular case of singular hyperbolicity.

\begin{definition}\label{def:sec-exp}
  A \emph{sectional hyperbolic set} is a singular hyperbolic one such that
    for every two-dimensional linear subspace
   $L_x \subset F_x$ one has
    \begin{align}\label{eq:def-sec-exp}
      \vert \det (DX_t \vert_{L_x})\vert > C e^{\la t},
      \forall t>0.
    \end{align}
  \end{definition}

\subsection{Singular hyperbolicity of various orders}\label{sec:p-sing-hyp}
\hfill

Given $E$ a vector space, we denote by $\wedge^p E$ the exterior power of order $p$ of $E$, defined as follows.
If $v_1,\dots, v_n$ is a basis of $E$ then $\wedge^p E$ is generated by
$\{v_{i_1}\wedge \cdots \wedge v_{i_p}\}_{1 \leq i \leq n, i_j \neq i_k, j \neq k}$.
Any linear transformation $A:E\to F$ induces a transformation $\wedge^p A:\wedge^p E\to\wedge^p F$.
Moreover, $v_{i_1}\wedge \cdots \wedge v_{i_p}$ can be viewed as the $p$-plane generated by
$\{v_{i_1}, \cdots, v_{i_p}\}$ if $i_j \neq i_k, j \neq k$.
As a reference for more information about exterior powers it is recommended \cite{A}, for instance.

We may define a new kind of singular hyperbolicity.

\begin{definition}\label{def:p-sing-hyp}
  A compact invariant set $\Lambda$ is $p$-singular hyperbolic (or $p$-sectionally hyperbolic) for a $C^1$ flow $X$
  if there exists a partially hyperbolic splitting $T_{\Lambda}M = E \oplus F$ such that $E$ is uniformly contracting and the central subbundle $F$ is $p$-sectionally expanding, with $2 \leq p \leq \dim(F)$.
\end{definition}

If $L_x$ is a $p$-plane, we can see it as $\widetilde{v}\in \wedge^p(F_x)\setminus \{0\}$ of norm one.
Hence, to obtain the singular expansion we just need to show that for some $\lambda>0$ and every $t>0$
holds the following inequality
$$\|\wedge^p DX_t(x).\widetilde{v}\|>Ce^{\lambda t}.$$

Our first main result concerns in a characterization of singular hyperbolicity of any order
via infinitesimal Lyapunov functions,
following \cite{ArSal2012}, \cite{ArSal2015}, \cite{BurnKatok94}, \cite{Pota79}, \cite{Wojtk85}, \cite{Wojtk01}.

Recall that, if $T: Z \to Z$ is a measurable map, we say that a probability measure $\mu$ is an invariant measure of $T$, if $\mu(T^{-1}(A)) = \mu(A)$, for every measurable set $A \subset Z$.
We say that $\mu$ is an invariant measure of $X$ if it is an invariant measure of $X_t$ for every $t \in \mathbb{R}$.
We will denote by $\mathcal{M}_X$ the set of all invariant measures of $X$.
A subset $Y\subset Z$ has \emph{total probability} if for every $\mu\in \mathcal{M}_X$ we have $\mu(Y)=1$ (see \cite{Man82}).

\begin{theorem}
\label{mthm:sec-hyp-equiv}
A compact invariant set $\Lambda$ whose singularities are hyperbolic (with $\indi \geq \indi(\J)$)
for $X \in \Mundo$ is a $p$-singular hyperbolic set if, and only if,
there exist a neighborhood $U$ of $\Lambda$ and a field of non-degenerate quadratic forms $\J$ on $U$
with index $1 \leq \indi(\J) \leq n-2$ such that $X$ is non-negative strictly $\J$-separated and
the spectrum of the diagonalized operator $DX_t$ satisfies the properties:
\begin{enumerate}
\item{} $r_1^- < 1$; and
\item{} $\Pi_{1}^{p} \  r_i^+ > 1$, \ \textrm{where} \ $2 \leq p \leq \dim(M) - \indi(\J)$,
\end{enumerate}
in a total probability subset of $\Lambda$.
Moreover, if $r_i^+ \cdot r_j^+ > 1$, for all $1 \leq i, j, \leq p, i \neq j$, in a total probability set,
then $\Lambda$ is a sectional-hyperbolic set.
\end{theorem}

In \cite{ArSal2012}, the authors proved the next result about sectional hyperbolicity.

As a direct application of Theorem \ref{mthm:sec-hyp-equiv} and Theorem\ref{thm:lyap-exp-sing-val} in Section $2$,
we reobtain the next one, without the assumption on the singularities.

\begin{corollary}\cite[Theorem D]{ArSal2012}
\label{mcor:2-sec-exp-J-monot}
  Suppose that all singularities of the attracting set $\Lambda$ of $U$ are
  all of them sectional-hyperbolic with index $\indi(\sigma) \geq \indi(\Lambda)$.
  Then, $\Lambda$ is a sectional-hyperbolic set for $X_t$ if, and
  only if, there is a field of quadratic forms $\J$ with index equal to $\indi(\Lambda)$
  such that $X_t$ is a non-negative strictly $\J$-separated
  flow on $U$ and for each compact invariant subset $\Gamma$ in $\Lambda^*=\Lambda \setminus \sing(X)$
  the linear Poincar\'e flow is strictly $\J_0$-monotonous for some field of quadratic forms $\J_0$
  equivalent to $\J$.
\end{corollary}

  Thus, Theorem \ref{mthm:sec-hyp-equiv} is an improvement to this result, once it does not requires a priori
  sectional hyperbolicity on the singularities.

\vspace{0.1in}

In \cite{AraArbSal}, this author together with V. Araujo and A. Arbieto,
proved that the requirements in the definition of sectional hyperbolicity can be weakened,
demanding the domination property only over the singularities, because in this setting the splitting is in fact dominated.
More precisely, we proved the next result.

\begin{theorem}\cite[Theorem A]{AraArbSal}\label{mthm:domination-ararbsal}
  Let $\Lambda$ be a compact invariant set of $X$ such that every singularity in this set is
  hyperbolic. Suppose that there exists a continuous invariant
  splitting of the tangent bundle of $\Lambda$,
  $T_{\Lambda}M = E \oplus F$, where $E$ is uniformly contracted, $F$ is sectionally expanding and for some constants $C,\lambda > 0$ we have
\begin{align}
  \|DX_t\vert_{E_\sigma}\|\cdot\|DX_{-t}\vert_{F_\sigma}\| &< \label{eq:sing-domin}
  Ce^{-\lambda t} \quad\text{for all}\quad
  \sigma\in\Lambda\cap\sing(X)\textrm{ and $t\geq 0$}.
\end{align}
Then $T_{\Lambda}M = E \oplus F$ is a dominated splitting.
\end{theorem}

The study of conditions to a given splitting of the tangent bundle to have the domination property is an important
research line in the area of Dynamical Systems, see \cite{ArbSal}, \cite{AraPac2010}, \cite{BDV2004}.

Some progress in this context has been obtained for instance in \cite[Theorem A]{ArSal2015}, jointly
with V. Araujo, where we give a characterization for dominated splitting based on $k$-th exterior powers, where $k = \dim F$.

We note that if $E\oplus F$ is a $DX_t$-invariant
splitting of $T_\Gamma M$, with $\{e_1,\dots,e_\ell\}$
a family of basis for $E$ and $\{f_1,\dots,f_h\}$ a
family of basis for $F$, then $\widetilde F=\wedge^kF$
generated by $\{f_{i_1}\wedge\dots\wedge
f_{i_k}\}_{1\le i_1<\dots<i_k\le h}$ is naturally
$\wedge^kDX_t$-invariant by construction. In addition,
$\tilde E$ generated by $\{e_{i_1}\wedge\dots\wedge
e_{i_k}\}_{1\le i_1<\dots<i_k\le \ell}$ together with
all the exterior products of $i$ basis elements of $E$
with $j$ basis elements of $F$, where $i+j=k$ and
$i,j\ge1$, is also $\wedge^kDX_t$-invariant and,
moreover, $\widetilde E\oplus \widetilde F$ gives a
splitting of the $k$th exterior power $\wedge^k
T_\Gamma M$ of the subbundle $T_\Gamma M$.

\begin{theorem}\cite[Theorem A]{ArSal2015}\label{mthm:bivectparthyp}
  Let $T_\Gamma M=E_\Gamma\oplus F_\Gamma$ be a
  $DX_t$-invariant splitting over the compact
  $X_t$-invariant subset $\Gamma$ such that 
  $\dim F=k\ge2$.  
  Let
  $\widetilde F=\wedge^k F$ be the
  $\wedge^k DX_t$-invariant subspace generated by the
  vectors of $F$ and $\tilde E$ be the
  $\wedge^k DX_t$-invariant subspace such that $\widetilde
  E\oplus\widetilde F$ is a splitting of the $k$th exterior
  power $\wedge^k T_\Gamma M$ of the subbundle $T_\Gamma M$.

  Then $E\oplus F$ is a dominated splitting if, and only if,
  $\widetilde E\oplus \widetilde F$ is a dominated splitting
  for $\wedge^k DX_t$.
\end{theorem}

We note that the equivalence is only valid if $k = \dim F$.

Here, it is proved a similar result to \cite[Theorem A]{AraArbSal}, but now it is done on $p$-sectional hypothesis. Note that, in this case, it is no longer true without some more requirements on the combinations of the Lyapunov exponents of the subbundles (due Theorem \ref{mthm:bivectparthyp}), since for $p > 2$ we can have uniform contraction on $E$, $p$-sectional expansion on $F$ and none dominated splitting, as exemplified below.

In \cite[Example 3]{ArSal2015}, we have an example where even $E \oplus F$ being dominated we do not obtain
$\widetilde{E} \oplus \widetilde{F}$ dominated, for $k < \dim F$. The next example is a similar one.

\begin{example}\label{ex:ThmA}
  Theorem~\ref{mthm:bivectparthyp} does not hold if we take
  $c < \dim F$: consider $\sigma$ a hyperbolic fixed point for
  a vector field $X$ in a $4$-manifold such that
  $DX(\sigma)=\diag\{-3, 2, 4, 10\}$. The splitting
  $E=\RR\times\{0^3\}, F=\{0\}\times\RR^3$ is dominated and
  hyperbolic but, for $c=2<3=\dim F$ the splitting $\tilde
  E\oplus \tilde F$ of the exterior square is not
  dominated. Indeed, the eigenvalues for $\tilde F$ are $2+4
  = 6, 2+10 = 12, 4+10 =14$, and for $\tilde E $ the
  eigenvalues are $-3+2 = -1, -3+4 = 1, -3+10 = 7$, so we
  have an eigenvalue $7$ in $\tilde E$ strictly bigger than
  the eigenvalue $6$ along $\tilde F$.
\end{example}

We can see that even under the domination assumption over singularities, we have no longer the same result as Theorem \ref{mthm:domination-ararbsal}, it is enough to take the union of an isolated hyperbolic singularity with a periodic orbit displaying the features of the above example.

However, we might ask how this assumption worked out in \cite{ArSal2015} and \cite{AraArbSal}. In fact, within the accounts of the results contained therein it is obtained domination due $2$-sectional expansion together with the uniform contraction. The singular case requires domination on the singularities, once it is necessary matching the splitting.

Observing these results, we can get a characterization of domination property based on Lyapunov spectrum, without any other assumption on the singularities. This is the content of our next result.

\begin{theorem}\label{mthm:Lyapunov-domination}
  Let $\Lambda$ be a compact invariant set of $X$.
  Suppose that there exists a continuous invariant
  splitting of the tangent bundle of $\Lambda$,
  $T_{\Lambda}M = E \oplus F$. Then $T_{\Lambda}M = E \oplus F$ is a dominated splitting if, and only if, exists $\eta < 0$ for which
       \begin{align*}
         \liminf\limits_{t \to +\infty} \frac{1}{t} \log \vert DX_t\vert_{E_x}\vert - \limsup\limits_{t \to +\infty} \frac{1}{t}\log m(DX_t\vert_{F_x}) < \eta,
       \end{align*}
       in a total probability set of $\Lambda$.
\end{theorem}

By transitivity, we obtain the next corollary.

\begin{corollary}\label{mcor:equiv-domina-ext-expon}
Suppose the assumptions of Theorem \ref{mthm:bivectparthyp}. Then, $\widetilde E\oplus \widetilde F$ is a dominated splitting
  for $\wedge^k DX_t$ if, and only if, there exists $\eta < 0$ for which
       \begin{align*}
         \liminf\limits_{t \to +\infty} \frac{1}{t} \log \vert DX_t\vert_{E_x}\vert - \limsup\limits_{t \to +\infty} \frac{1}{t}\log m(DX_t\vert_{F_x}) < \eta,
       \end{align*}
       in a total probability set of $\Lambda$.
\end{corollary}

\subsection{p-sectional Lyapunov exponents}
\hfill

The next definition reminds a previous one from Arbieto \cite{arbieto2010} which deals with, in his terminology,
the sectional Lyapunov exponents.

Based in the same ideas, we can state an analogous term for general singular sets.

Inspired by \cite{arbieto2010}, we finally define:

\begin{definition}\label{def:sing-lyap-exp}
The \emph{p-sectional Lyapunov exponents} (or \emph{Lyapunov exponents of order $p$}) of $x$ along $F$ are the limits
$$\lim_{t\to+\infty}\frac{1}{t}\log\| \wedge^p DX_t(x).\widetilde{v}\|$$
whenever they exists, where $\widetilde{v}\in \wedge^p F_x-\{0\}$.
\end{definition}

Following the corresponding results from \cite[Theorem B]{AraArbSal} and \cite[Theorem 2.3]{arbieto2010},
just by some modifications in computations and hyphotesis, changing
$\|\wedge^2 DX_t(x).\widetilde{v}\|$ by $\|\wedge^p DX_t(x).\widetilde{v}\|$.

We obtain, via Theorem \ref{thm:lyap-exp-sing-val}, the analogous results for singular hyperbolic and partially hyperbolic sets of the main result of this paper.

\begin{corollary}\label{mcor:equiv-partial}
  Let $\Lambda$ be a compact invariant set of $X$ such that every singularity in this set is
  hyperbolic. There exists a continuous invariant
  splitting of the tangent bundle,
  $T_{\Lambda}M = E \oplus F$, of $\Lambda$ where:
  \begin{enumerate}
    \item the Lyapunov exponents on $E$ are negative (or positive on $F$), and
    \item $\liminf\limits_{t \to +\infty} \frac{1}{t} \log \vert DX_t\vert_{E_x}\vert - \limsup\limits_{t \to +\infty} \frac{1}{t}\log m(DX_t\vert_{F_x}) < 0$,
  \end{enumerate}
   in a total probability set of $\Lambda$, if and only if, $T_{\Lambda}M = E \oplus F$ is a partially hyperbolic splitting.
\end{corollary}

This way, we can extend and improve \cite[Theorem B]{AraArbSal} and \cite[Theorem 2.3]{arbieto2010}, as follow.

\begin{corollary}
\label{mcor:p-sing-lyap-exp}
Let $\Lambda$ a compact invariant set for a flow $X_t$ such that every singularity $\sigma \in \Lambda$ is hyperbolic.
Suppose that there is a continuous invariant splitting $T_{\Lambda}M=E\oplus F$.
The set $\Lambda$ is $p$-singular hyperbolic for the flow if, and only if, on a set of total probability in $\Lambda$,
\begin{enumerate}
\item $\liminf\limits_{t \to +\infty} \frac{1}{t} \log \vert DX_t\vert_{E_x}\vert - \limsup\limits_{t \to +\infty} \frac{1}{t}\log m(DX_t\vert_{F_x}) < 0$,
\item the Lyapunov exponents in the $E$ direction are negative and
\item the $p$-sectional Lyapunov exponents in the $F$ direction are positive .
\end{enumerate}
\end{corollary}

Hence, the definition of singular hyperbolicity (of any order, including the classical one) can be rewritten based on the Lyapunov exponents.

\begin{definition}\label{new-def-singhyp}
  A compact invariant set $\Lambda \subset M$ is
  \emph{$p$-singular hyperbolic} for $X$ if
  all singularities in $\Lambda$ are hyperbolic, there
  are a continuous invariant splitting of the tangent bundle on
  $T_{\Lambda}M = E \oplus F$ and constants $C,\lambda > 0$ such that for every $x
  \in \Lambda$ and every $t>0$ we have
\begin{enumerate}
\item the Lyapunov exponents in the $E$ direction are negative,
\item the $p$-sectional Lyapunov exponents in the $F$ direction are positive,
\item $\liminf\limits_{t \to +\infty} \frac{1}{t} \log \vert DX_t\vert_{E_x}\vert - \limsup\limits_{t \to +\infty} \frac{1}{t}\log m(DX_t\vert_{F_x}) < 0$
\end{enumerate}
 in a total probability set of $\Lambda$.
\end{definition}

The last item guarantees that the dominated splitting of the singularities matches to the one over the remainder of $\Lambda$.

\vspace{0.1in}

\begin{remark}
  \label{rmk:sec-exp-discrete}
  The properties of $p$-singular hyperbolicity can be expressed in the following
  equivalent forms; see \cite{AraPac2010} for the classical one.  There exists
  $T>0$ such that
  \begin{itemize}
  \item $\|DX^T\vert_{E_x}\|<\frac12$ for all $x\in\Lambda$
    (uniform contraction); and
  \item $|\wedge^p (DX^T\vert_{\wedge^p F_x})|> 2$ for all $x\in\Lambda$.
  \end{itemize}
\end{remark}

From now on, we consider $M$ a connected compact finite
dimensional riemannian manifold and all singularities of
$X$ (if they exist) are hyperbolic.

\section{Fields of quadratic forms}
\label{sec:fields-quadrat-forms}
\hfill

In this section, we introduce the quadratic forms and its properties.

Let $\J:E_U\to\RR$ be a continuous family of quadratic forms
$\J_x:E_x\to\RR$ which are non-degenerate and have index
$0<q<\dim(E)=n$, where $U\subset M$ is an open set such that $X_t(U) \subset \overline{U}$
for a vector field $X$. We also assume that $(\J_x)_{x\in U}$ is
continuously differentiable along the flow.

The continuity assumption on $\J$ just means that for every
continuous section $Z$ of $E_U$ the map $U\to\RR$ given by
$x\mapsto \J(Z(x))$ is continuous. The $C^1$ assumption on
$\J$ along the flow means that the map $x\mapsto
\J_{X_t(x)} (Z(X_t(x)))$ is continuously differentiable for
all $x\in U$ and each $C^1$ section $Z$ of $E_U$.

The assumption that $M$ is a compact manifold enables us to
globally define an inner product in $E$ with respect to
which we can find the an orthonormal basis associated to
$\J_x$ for each $x$, as follows. Fixing an orthonormal basis
on $E_x$ we can define the linear operator
\begin{align*}
  J_x:E_x\to E_x \quad\text{such that}\quad \J_x(v)=<J_x
  v,v> \quad \text{for all}\quad v\in T_xM,
\end{align*}
where $<,>=<,>_x$ is the inner product at $E_x$. Since we
can always replace $J_x$ by $(J_x+J_x^*)/2$ without changing
the last identity, where $J_x^*$ is the adjoint of $J_x$
with respect to $<,>$, we can assume that $J_x$ is
self-adjoint without loss of generality.  Hence, we
represent $\J(v)$ by a non-degenerate symmetric bilinear
form $<\J_x v,v>_x$. Now we use Lagrange's method to
diagonalize this bilinear form, obtaining a base
$\{u_1,\dots,u_n\}$ of $E_x$ such that
\begin{align*}
  \J_x(\sum_{i}\alpha_iu_i)=\sum_{i=1}^q -\lambda_i\alpha_i^2 +
  \sum_{j=q+1}^n \lambda_j\alpha_j^2, \quad
  (\alpha_1,\dots,\alpha_n)\in\RR^n.
\end{align*}
Replacing each element of
this base according to $v_i=|\lambda_i|^{1/2}u_i$ we deduce that
\begin{align*}
\J_x(\sum_{i}\alpha_iv_i)=\sum_{i=1}^q -\alpha_i^2 +
  \sum_{j=q+1}^n \alpha_j^2, \quad
  (\alpha_1,\dots,\alpha_n)\in\RR^n.
\end{align*}
Finally, we can redefine $<,>$ so that the base
$\{v_1,\dots, v_n\}$ is orthonormal. This can be done
smoothly in a neighborhood of $x$ in $M$ since we are
assuming that the quadratic forms are non-degenerate; the
reader can check the method of Lagrange in a standard Linear
Algebra textbook and observe that the steps can be performed
with small perturbations, for instance in \cite{Maltsev63}.

In this adapted inner product we have that $J_x$ has entries
from $\{-1,0,1\}$ only, $J_x^*=J_x$ and also that
$J_x^2=J_x$.

Having fixed the orthonormal frame as above, the
\emph{standard negative subspace} at $x$ is the one spanned
by $v_{1},\dots, v_{q}$ and the \emph{standard positive
  subspace} at $x$ is the one spanned $v_{q+1},\dots,v_n$.

\subsubsection{Positive and negative cones}
\label{sec:positive-negative-co}

Let $\cC_\pm=\{C_\pm(x)\}_{x\in U}$ be the family of
positive and negative cones
\begin{align*}
  C_\pm(x):=\{0\}\cup\{v\in E_x: \pm\J_x(v)>0\}  \quad x\in U
\end{align*}
and also let $\cC_0=\{C_0(x)\}_{x\in U}$ be the corresponding
family of zero vectors $C_0(x)=\J_x^{-1}(\{0\})$ for all
$x\in U$.
In the adapted coordinates obtained above we have
\begin{align*}
  C_0(x)=\{v=\sum_{i}\alpha_iv_i\in E_x :
  \sum_{j=q+1}^n \alpha_j^2 = \sum_{i=1}^q
  \alpha_i^2\}
\end{align*}
is the set of \emph{extreme points} of $C_\pm(x)$.

The following definitions are fundamental to state our main
result.

\begin{definition}
\label{def:J-separated}
Given a continuous field of non-degenerate quadratic forms
$\J$ with constant index on the trapping region $U$ for the
flow $X_t$, we say that the flow is
\begin{itemize}
\item $\J$-\emph{separated} if $DX_t(x)(C_+(x))\subset
  C_+(X_t(x))$, for all $t>0$ and $x\in U$;
\item \emph{strictly $\J$-separated} if $DX_t(x)(C_+(x)\cup
  C_0(x))\subset C_+(X_t(x))$, for all $t>0$ and $x\in U$;
\item $\J$-\emph{monotone} if $\J_{X_t(x)}(DX_t(x)v)\ge \J_x(v)$, for each $v\in
  T_xM\setminus\{0\}$ and $t>0$;
\item \emph{strictly $\J$-monotone} if $\partial_t\big(\J_{X_t(x)}(DX_t(x)v)\big)\mid_{t=0}>0$,
  for all $v\in T_xM\setminus\{0\}$, $t>0$ and $x\in U$;
\item $\J$-\emph{isometry} if $\J_{X_t(x)}(DX_t(x)v) = \J_x(v)$, for each $v\in T_xM$ and $x\in U$.
\end{itemize}
\end{definition}
Thus, $\J$-separation corresponds to simple cone invariance
and strict $\J$-separation corresponds to strict cone
invariance under the action of $D_t(x)$.

\begin{remark}\label{rmk:J-separated-C-}
  If a flow is strictly $\J$-separated, then for $v\in T_xM$
  such that $\J_x(v)\le0$ we have
  $$
  \J_{X_{-t}(x)}(DX_{-t}(v))<0,
  $$
  for all $t>0$ and $x$ such that $X_{-s}(x)\in U$ for every $s\in[-t,0]$.

  Indeed, otherwise $\J_{X_{-t}(x)}(DX_{-t}(v))\ge0$ would imply
  $\J_x(v)=\J_x\big(DX_t(DX_{-t}(v))\big)>0$, contradicting
  the assumption that $v$ was a non-positive vector.

  This means that a flow $X_t$ is strictly
    $\J$-separated if, and only if, its time reversal
    $X_{-t}$ is strictly $(-\J)$-separated.
\end{remark}

A vector field $X$ is $\J$-\emph{non-negative} on $U$ if
$\J(X(x))\ge0$ for all $x\in U$, and
$\J$-\emph{non-positive} on $U$ if $\J(X(x))\leq 0$ for all
$x\in U$. When the quadratic form used in the context is
clear, we will simply say that $X$ is non-negative or
non-positive.

We apply this notion to the linear Poincar\'e flow defined on
regular orbits of $X_t$ as follows.

We assume that the vector field $X$ is non-negative on $U$.
Then, the span $E^X_x$ of
$X(x)\neq 0$ is a $\J$-non-degenerate subspace.
According to item (1) of Proposition~\ref{pr:propbilinear}, this means
that $T_xM=E_x^X\oplus N_x$, where $N_x$ is the
pseudo-orthogonal complement of $E^X_x$ with respect to the
bilinear form $\J$, and $N_x$ is also
non-degenerate. Moreover, by the definition, the index of $\J$
restricted to $N_x$ is the same as
the index of $\J$. Thus, we can define on $N_x$ the
cones of positive and negative
vectors, respectively, $N_x^+$ and $N_x^-$, just like before.

Now we define the Linear Poincar\'e Flow $P^{\, t}$ of $X_t$
along the orbit of $x$, by projecting $DX_t$ orthogonally
(with respect to $\J$) over $N_{X_t(x)}$ for each $t\in\RR$:
\begin{align*}
  P^{\, t} v := \Pi_{X_t(x)}DX_t v ,
  \quad
  v\in T_x M, t\in\RR, X(x)\neq 0,
\end{align*}
where $\Pi_{X_t(x)}:T_{X_t(x)}M\to N_{X_t(x)}$ is the
projection on $N_{X_t(x)}$ parallel to $X(X_t(x))$.

We remark that the definition of $\Pi_x$ depends on $X(x)$ and
$\J_X$ only. The linear Poincar\'e flow $P^{\,t}$ is a linear
multiplicative cocycle over $X_t$ on the set $U$ with the
exclusion of the singularities of $X$.

In this setting we can say that the linear Poincar\'e flow is
$\J$-separated and $\J$-monotonous
using the non-degenerate bilinear form $\J$ restricted to
$N_x$ for a regular $x\in U$.

More precisely: $P^t$ is $\J$-monotonous if
$\partial_t\J(P^tv)\mid_{t=0}\ge0$, for
each $x\in U, v\in T_xM\setminus\{0\}$ and $t>0$,
and strictly $\J$-monotonous if $\partial_t\J(P^tv)\mid_{t=0}>0$,
for all $v\in T_xM\setminus\{0\}$, $t>0$ and $x\in U$.

\begin{proposition}
  \label{pr:J-separated-spectrum}
  Let $L:V\to V$ be a $\J$-separated linear operator. Then
  \begin{enumerate}
  \item $L$ can be uniquely represented by $L=RU$, where $U$
    is a $\J$-isometry and
    $R$ is $\J$-symmetric (or $\J$-pseudo-adjoint; see
    Proposition~\ref{pr:propbilinear}) with positive
    spectrum.
  \item the operator $R$ can be diagonalized by a
    $\J$-isometry. Moreover the eigenvalues of $R$ satisfy
    \begin{align*}
      0<r_-^q\le\dots\le r_-^1=r_-\le r_+=r_1^+\le\dots\le r_+^p.
    \end{align*}
  \item the operator $L$ is (strictly) $\J$-monotonous if,
    and only if, $r_-\le (<) 1$ and $r_+\ge (>) 1$.
  \end{enumerate}
\end{proposition}

\subsection{J-separated linear maps}
\label{sec:j-separat-linear}

\subsubsection{J-symmetrical matrixes and J-selfadjoint operators}
\label{sec:j-symmetr-matrix}

The symmetrical bilinear form defined by
$$
(v,w)=\langle J_x v,w\rangle,
$$
$v,w\in E_x$ for $x\in M$ endows
$E_x$ with a pseudo-Euclidean structure.

Since $\J_x$ is non-degenerate, then the form $(\cdot,\cdot)$ is likewise
non-degenerate and many properties of inner products are
shared with symmetrical non-degenerate bilinear forms. We
state some of them below.

\begin{proposition}
  \label{pr:propbilinear}
  Let $(\cdot,\cdot):V\times V \to\RR$ be a real symmetric
  non-degenerate bilinear form on the real finite
  dimensional vector space $V$.
  \begin{enumerate}
  \item $E$ is a subspace of $V$ for which $(\cdot,\cdot)$ is
    non-degenerate if, and only if, $V=E\oplus E^\perp$.

    We recall that $E^\perp:=\{v\in V: (v,w)=0
    \quad\text{for all}\quad w\in E\}$, the
    pseudo-orthogonal space of $E$, is defined using the
    bilinear form.
  \item Every base $\{v_1,\dots,v_n\}$ of $V$ can be
    orthogonalized by the usual Gram-Schmidt process of
    Euclidean spaces, that is, there are linear combinations
    of the basis vectors $\{w_1,\dots, w_n\}$ such that they
    form a basis of $V$ and
    $(w_i,w_j)=0$ for $i\neq j$.  Then this last base can be
    pseudo-normalized: letting $u_i=|(w_i,w_i)|^{-1/2}w_i$ we
    get $(u_i,u_j)=\pm\delta_{ij}, i,j=1,\dots,n$.
  \item There exists a maximal dimension $p$ for a subspace
    $P_+$ of $\J$-positive vectors and a maximal dimension
    $q$ for a subspace $P_-$ of $\J$-negative vectors;
    we have $p+q=\dim V$ and $q$ is known
    as the \emph{index} of $\J$.
  \item For every linear map $L:V\to\RR$ there exists a
    unique $v\in V$ such that $L(w)=(v,w)$ for each $w\in V$.
  \item For each $L:V\to V$ linear there exists a unique
    linear operator $L^+:V\to V$ (the pseudo-adjoint) such that
    $(L(v),w)=(v,L^+(w))$ for every $v,w\in V$.
  \item Every pseudo-self-adjoint $L:V\to V$, that is,
    such that $L=L^+$, satisfies
    \begin{enumerate}
    \item eigenspaces corresponding to distinct eigenvalues
      are pseudo-orthogonal;
    \item if a subspace $E$ is $L$-invariant, then $E^\perp$
      is also $L$-invariant.
    \end{enumerate}
  \end{enumerate}
\end{proposition}

The proofs are rather standard and can be found in
\cite{Maltsev63}.

The following simple result will be very useful in what follows.

\begin{lemma}
  \label{le:kuhne}
Let $V$ be a real finite dimensional vector space endowed
with a non-positive definite and non-degenerate quadratic
form $\J:V\to\RR$.

If a symmetric bilinear form $F:V\times V\to\RR$ is
non-negative on $C_0$ then
\begin{align*}
  r_+=\inf_{v\in C_+} \frac{F(v,v)}{\langle Jv,v\rangle}
  \ge \sup_{u\in C_-}\frac{F(u,u)}{\langle Ju,u\rangle}=r_-
\end{align*}
and for every $r$ in $[r_-,r_+]$ we have
$F(v,v)\ge r\langle Jv,v\rangle$ for each vector $v$.

In addition, if $F(\cdot,\cdot)$ is positive on
$C_0\setminus\{0\}$, then $r_-<r_+$ and $F(v,v)>
r\langle Jv,v\rangle$ for all vectors $v$ and $r\in(r_-,r_+)$.
\end{lemma}

\begin{remark}
  \label{rmk:Jseparated}
  Lemma~\ref{le:kuhne} shows that if $F(v,w)=\langle \tilde J
  v,w\rangle$ for some self-adjoint operator $\tilde J$ and
  $F(v,v)\ge0$ for all $v$ such that $\langle J v,
  v\rangle=0$, then we can find $a\in\RR$ such that
  $\tilde J \ge a J$. This means precisely that $\langle
  \tilde J v,v\rangle\ge a\langle Jv, v\rangle$ for all
  $v$.

  If, in addition, we have $F(v,v)>0$ for all $v$ such that
  $\langle J v, v\rangle=0$, then we obtain a strict
  inequality $\tilde J > a J$ for some $a\in\RR$ since the
  infimum in the statement of Lemma~\ref{le:kuhne} is
  strictly bigger than the supremum.
\end{remark}

The (longer) proofs of the following results can be found
in~\cite{Wojtk01} or in~\cite{Pota79}; see also~\cite{Wojtk09}.

For a $\J$-separated operator $L:V\to V$ and a
$d$-dimensional subspace $F_+\subset C_+$, the subspaces
$F_+$ and $L(F_+)\subset C_+$ have an inner product given by
$\J$. Thus both subspaces are endowed with volume
elements. Let $\alpha_d(L;F_+)$ be the rate of expansion of
volume of $L\mid_{F_+}$ and $\sigma_d(L)$ be the infimum of
$\alpha_d(L;F_+)$ over all $d$-dimensional subspaces $F_+$
of $C_+$.

\begin{proposition}
  \label{pr:product-vol-exp}
  We have $\sigma_d(L)=r_+^1 \cdots r_+^d$, where $r^i_+$
  are given by Proposition~\ref{pr:J-separated-spectrum}(2).

  Moreover, if $L_1,L_2$ are $\J$-separated, then
  $\sigma_d(L_1L_2)\ge\sigma_d(L_1)\sigma_d(L_2)$.
\end{proposition}

The following corollary is very useful.

\begin{corollary}
  \label{cor:compos-max-exp}
  For $\J$-separated operators $L_1,L_2:V\to V$ we have
  \begin{align*}
    r_+^1(L_1L_2)\ge r_+^1(L_1) r_+^1(L_2) \quad\text{and}\quad
    r_-^1(L_1L_2)\le r_-^1(L_1)r_-^1(L_2).
  \end{align*}
  Moreover, if the operators are strictly $\J$-separated,
  then the inequalities are strict.
\end{corollary}

\begin{remark}\label{rmk:J-mon-spec}
Another important property about the singular values of a $\J$-separated operator $L$ is that
$$  r_+^1 = r_+ \ge 1 (> 1) \quad\text{and}\quad r_-^1 = r_- \le 1 (< 1)$$
if, and only if, $L$ is (strictly) $\J$-monotone.

This property will be used a lot of times in our proofs.
\end{remark}

\subsection{Lyapunov exponents}
\hfill

It is well known that under conditions of measurability, by Oseledec's Ergodic Theorem, there exist a full probability set $X$ such that for every $x \in Y$ there is an invariant decomposition
\begin{align*}
T_xM = \langle X\rangle \oplus E_{1}(x) \oplus \cdots \oplus E_{l(x)}(x)
\end{align*}
and numbers $\chi_1 < \cdots < \chi_l$ corresponding to the limits
\begin{align*}
  \chi_j = \lim\limits_{t \to +\infty} \frac{1}{t} \log \Vert DX_t(x) \cdot v\Vert,
\end{align*}
for every $v \in E_i(x)\setminus \{0\}, i = 1, \cdots, l(x)$.

In this setting, Wojtkowski \cite{Wojtk01} proved that the logarithm of the pseudo-Euclidean singular values $0 \leq r_q^- \leq \cdots \leq r_1^- \leq r_1^+ \leq \cdots \leq r_p^+$ of $DX_t$ are $\mu$-integrable, and obtained estimates of the Lyapunov exponents related to the singular eigenvalues of strictly $\J$-separated maps.

\begin{theorem}\cite[Corollary 3.7]{Wojtk01}
\label{thm:lyap-exp-sing-val}

For $1 \leq k_1 \leq p$  and $1 \leq k_2 \leq q$

\begin{align*}
\chi^-_1 + \cdots + \chi^-_{k_1} \leq \sum_{i=1}^{k_1} \int \log r^-_i d\mu \ \textrm{and} \ \chi^+_1 + \cdots + \chi^+_{k_2} \geq \sum_{i=1}^{k_2} \int \log r^+_i d\mu.
\end{align*}

\end{theorem}

Look that, if $X_t$ is a $\J$-separated flow on $\Lambda$, for each diffeomorphism $DX_t$ if we fix $t > 0$, the last theorem holds for $r^{\pm, t}_i$, where $r^{\pm, t}_i$ are the singular $\J$-values of $DX_t$.

\section{Proof of Theorems}
\hfill

In this section, we prove our mains results.

First, we prove Theorem \ref{mthm:sec-hyp-equiv}, by using Corollary \ref{mcor:p-sing-lyap-exp} which is proved below.

\begin{proof}[Proof of Theorem \ref{mthm:sec-hyp-equiv}]
Suppose $\Lambda$ $p$-singular hyperbolic set of index $\indi$. Then, $1 \leq \indi \leq n-2$ and there is a dominated splitting $T_{\Lambda}M = E \oplus F$, where $E$ is uniformly contracting and $F$ is uniformly $p$-sectionally expanding. Moreover, $\langle X \rangle \subset F$, by Lemma \ref{le:flow-center}. By using adapted metric \cite{Goum07}, we construct the quadratic forms $\J$ such that $X$ is non-negative strictly $\J$-separated.
By Proposition \ref{pr:J-separated-spectrum} and Corollary \ref{cor:compos-max-exp}, there is a $\J$-diagonalization of $DX_t$ by a $\J$-isometry, that we are also denoting by $DX_t$, such that its spectrum has the required properties. In fact, for each singular value $r_i^-$ corresponding to the contracting subspace, we must have $r_i^- < 1$. Analogously, as $F$ is a $p$-sectionally expanding subbundle, the sum of each $p$  corresponding singular value, $r_{i_1}^+, \cdots, r_{i_p}^+$, must be greater than one. Even including the corresponding field direction.

Reciprocally, suppose that in a total probability subset of $\Lambda$ we have $r_1^-<1$ and $\Pi_{j=1}^{p} \  r_{i_j}^+ > 1$, \ \textrm{where} \ $2 \leq p \leq \dim(M) - \indi(\J)$.

Moreover, strictly $\J$-separation guarantees that there exists a dominated splitting.
Let $T_{\Lambda}M = E \oplus F$ the corresponding splitting and the decomposition in direct sum of Lyapunov subspaces
\begin{align*}
  E_x = \oplus_{j=0}^{r} E_j(x), F_x = \oplus_{j=0}^{s(x)-1} F_j(x).
\end{align*}

By Theorem \ref{thm:lyap-exp-sing-val},
\begin{align*}
\chi^-_1 + \cdots + \chi^-_{r} \leq \sum_{i=1}^{r} \int \log r^-_i d\mu \ \textrm{and} \ \chi^+_{i_0} + \cdots + \chi^+_{i_p} \geq \sum_{j=1}^{p} \int \log r^+_{i_j} d\mu.
\end{align*}
So, we obtain that the Lyapunov exponents over $E$ are all of them negative and the $p$-sectional Lyapunov exponents on $F$ are all of them positive, in a total probability subset. Now, Theorem \ref{mthm:Lyapunov-domination} and Corollary \ref{mcor:p-sing-lyap-exp} imply that $\Lambda$ is a $p$-singular hyperbolic set for $X$.
\end{proof}

We recall now that, fixed a compact $X_t$-invariant subset $\Lambda$, we say
that a family of functions $\{f_t:\Lambda\to \RR\}_{t\in \RR}$
is subadditive if for every $x\in M$ and $t,s\in \RR$ we
have that $f_{t+s}(x)\leq f_s(x)+f_t(X_s(x))$.

\begin{proof}[Proof of Theorem \ref{mthm:Lyapunov-domination}]

Note that, once $T_{\Lambda}M = E \oplus F$ is a dominated splitting, there is an indefinite $C^1$ field of quadratic forms $\J$ such a way $X$ is strictly separated and, by Proposition \ref{pr:J-separated-spectrum},
\begin{align*}
      0<r_-^q\le\dots\le r_-^1=r_- < r_+=r_1^+\le\dots\le r_+^p.
    \end{align*}

Moreover, by Corollary \ref{thm:lyap-exp-sing-val},
\begin{align*}
\chi^-_1 + \cdots + \chi^-_{k_1} \leq \sum_{i=1}^{k_1} \int \log r^-_i d\mu \ \textrm{and} \ \chi^+_1 + \cdots + \chi^+_{k_2} \geq \sum_{i=1}^{k_2} \int \log r^+_i d\mu.
\end{align*}
Since $r_- - r_+ < 0$, we obtain
\begin{align*}
   \liminf\limits_{t \to +\infty} \frac{1}{t} \log \vert DX_t\vert_{E_x}\vert - \limsup\limits_{t \to +\infty} \frac{1}{t}\log m(DX_t\vert_{F_x}) =\\
   = \max\{\chi_i^E(x),1\le i\le r(x)\}
  -
  \min\{\chi_i^F(x),1\le i\le s(x)\}
  \le \eta <0,
\end{align*}
for all $x \in \Lambda$, in particular, in a total probability set.

Reciprocally, suppose that there exists a continuous invariant decomposition, $T_\Lambda M = E \oplus F$, and $\eta < 0$ such that
\begin{align*}
   \liminf\limits_{t \to +\infty} \frac{1}{t} \log \vert DX_t\vert_{E_x}\vert - \limsup\limits_{t \to +\infty} \frac{1}{t}\log m(DX_t\vert_{F_x})\le \eta <0,
\end{align*}
in a total probability set in $\Lambda$.

Consider $f_t(x)=\log \frac{\| DX_t|{E_x}\|}{m(DX_t\mid
  F_x)}$, which is a subadditive family of continuous functions and satisfies
  \begin{align*}
  \overline{f}(x)
  &=
  \liminf_{t\to+\infty}\frac{f_t(x)}{t}\le
  \liminf_{n\to+\infty}\frac1t\log\|
  DX_t|{E_x}\|
  -\limsup_{n\to+\infty}\frac1t\log m(
  DX_t|{F_x})\le \eta < 0.
\end{align*}
By Subadittive Ergodic Theorem \cite{Ki}, the function $\overline{f}(x)=\liminf\limits_{t\to+\infty}\frac{f_t(x)}{t}$ coincides with $\widetilde{f}(x)=\lim\limits_{t\to+\infty}\frac{1}{t}f_t(x)$ in a set of total probability. Moreover, for any invariant measure $\mu$ we have that $\int \widetilde{f}d\mu=\lim\limits_{t\to+\infty}\int \frac{f_t}{t}d\mu$.

Thus, we can use the following result from \cite{ArbSal}.

\begin{proposition}\cite[Corollary 4.2]{ArbSal}\label{prop:subadd}
Let $\{t\mapsto f_t:S\to \RR\}_{t\in \RR}$ be a continuous family of continuous functions which is subadditive and suppose that $\int \widetilde{f}(x) d\mu < 0$ for every $\mu\in \mathcal{M}_X$, with $\widetilde{f}(x):=\lim\limits_{t\to+\infty}\frac{1}{t}f_t(x)$. Then there exist a $T > 0$ and a constant $\eta < 0$ such that for every $x\in S$ and every $t \geq T$:
$$f_t(x) \leq  \eta t.$$
\end{proposition}

Note that, all of the last accounts are true independently to $x$ is either a regular or a singular point.

Hence, we obtain $f_t(x) \leq k - \eta t, t \geq 0, x \in \Lambda$, for some constant $k > 0$, and this gives us the domination property on $\Lambda$.

\end{proof}

Now, we prove the Corollary \ref{mcor:equiv-partial}.

\begin{proof}[Proof of Corollary \ref{mcor:equiv-partial}]

Suppose that we are under the hypothesis.

By Theorem \ref{mthm:Lyapunov-domination}, $E \oplus F$ is a dominated splitting on $\Lambda$.

Since $E$ is an invariant subbundle, consider $f_t(x) = \log \Vert DX_t\vert_{E_{x}}\Vert, t\in \RR$, as our subadditive family.

As in the proof of Theorem \ref{mthm:Lyapunov-domination}, we obtain $f_t(x) \leq k - \eta t, t \geq 0, x \in \Lambda$, for some constant $k > 0$. This means that $E$ is uniformly contracting under the action of $DX_t$.

The case of positive Lyapunov exponents over $F$ is analogous, by taking $f_t(x) = \log \Vert DX_{-t}\vert_{F_x} $. (Also see proof of \cite[Theorem B]{AraArbSal}).

For the converse, by using adapted metrics (as in the proof \cite[Theorem A]{ArSal2012}) we obtain a $C^1$ field $\J$ of nondegenerate quadratic forms for which $X$ is nonnegative strictly separated. Now, Proposition \ref{pr:J-separated-spectrum} and Theorem \ref{thm:lyap-exp-sing-val} complete the proof.
\end{proof}

Finally, the proof of Corollary \ref{mcor:p-sing-lyap-exp}.

We also need to use the following lemma.

Let $\Lambda$ be a compact invariant set for a flow
  $X$ of a $C^1$ vector field $X$ on $M$.
\begin{lemma} \cite{AraArbSal}
  \label{le:flow-center}
  Given a continuous splitting $T_\Lambda M = E\oplus F$
  such that $E$ is uniformly contracted, then $X(x)\in F_x$ for all $x\in \Lambda$.
\end{lemma}

\begin{proof}[Proof of Corollary \ref{mcor:p-sing-lyap-exp}]

By Theorem \ref{mthm:Lyapunov-domination}, $T_\Lambda M = E \oplus F$ is a dominated splitting.

If $x=\sigma \in \sing(X)$, by hyperbolicity, we obtain the desired features.

Following Corollary \ref{mcor:equiv-partial}, we obtain that this is a partially hyperbolic splitting as well, with subbundle $E$ uniformly contracting.
By Lemma \ref{le:flow-center}, if $x$ is a regular point, the flow direction $E^X(x)$ is contained in $F(x)$.

Since $F$ is an invariant subbundle, consider $f_t(x) = \log \Vert \wedge^p DX_t\vert_{F_{x}}\Vert, t\in \RR$,
and there is a decomposition in direct sum of Lyapunov subspaces
\begin{align*}
  F_x = \oplus_{j=0}^{s(x)-1} F_j(x).
\end{align*}
  One of them, say $E^X = F_0(x)$, generated by $X(x) \neq 0$. Denote by $\chi_j^F(x), j=1, \cdots s(x)-1$
  the corresponding Lyapunov exponents.

   Fixing $i_1, \cdots, i_p \in \{1, \cdots, s(x)-1\}$ and
   considering $p$ vectors $v_1 \in F_{i_1}\setminus\{0\}, \cdots v_{p-1} \in F_{i_{p-1}}\setminus\{0\}$,
   put $L = span \{X(x), v_1, \cdots, v_{p-1}\}$ as the generated $p$-plane.

  From assumption,
  \begin{align*}
    0 < \chi \leq \liminf\limits_{t \to +\infty} \frac{1}{t} \log \vert \wedge^p DX_t\vert_L \vert = \chi_0^F + \chi_{i_1}^F +\cdots +\chi_{i_{p-1}}^F,
  \end{align*}
  and we obtain
  \begin{align*}
    \sum_{j=1}^{p}\chi_{i_j}^F \geq \chi > 0, \forall i_j \in \{1, \cdots, s(x)-1\}.
  \end{align*}

For some singularity, $\sigma \in \Lambda$, we must have $\overline{f(\sigma)} \leq - \chi$,
as a consequence of domination .

Now applying the following proposition from \cite{arbieto2010}:

\begin{proposition}
\label{prop3.4-arbieto}
Let $\{t\mapsto f_t:\Lambda\to \RR\}_{t\in \RR}$ be a
continuous family of continuous function which is
subadditive and suppose that $\ov{f}(x)<0$ in a set of
total probability. Then there exist constants $C>0$ and
$\lambda<0$ such that for every $x\in \Lambda$ and every
$t>0$ we have
$\exp(f_t(x))\leq C \exp(\frac{\lambda t}{2}),$
\end{proposition}

to the function $f_t(x)$ give us constants $D > 0$ and
$\eta < 0$ for which $\|\wedge^p DX_{-t}|_{\wedge^p F_{X_t(x)}}\|\le
De^{\eta t}$, so $F$ is a $p$-sectionally expanding subbundle.

The converse follows from the lines of the last proof, by using Proposition \ref{pr:J-separated-spectrum} and Theorem \ref{thm:lyap-exp-sing-val}. So, we are done.

\end{proof}

\bibliographystyle{amsplain}

\end{document}